\documentclass[12pt]{amsart}
\usepackage[english]{babel}
\usepackage{amsmath}
\usepackage{amsthm}
\usepackage{amssymb}
\usepackage{mathrsfs}
\usepackage{enumerate}
\usepackage[notcite, final, notref]{showkeys}
\usepackage{dsfont}

\usepackage[dvips]{color}
\newtheorem{theorem}{Theorem}[section]

\newtheorem{lemma}[theorem]{Lemma}
\newtheorem{proposition}[theorem]{Proposition}

\newtheorem{definition}[theorem]{Definition}

\theoremstyle{definition}
\newtheorem{remark}[theorem]{Remark}

\usepackage{xcolor}

\title[Beurling-Lax type theorems]{Beurling-Lax type theorems in the complex and quaternionic setting: the half-space case}

\author[D. Alpay]{Daniel Alpay}
\address{(DA) Department of Mathematics\\
Ben-Gurion University of the Negev\\
Beer-Sheva 84105 Israel}
\email{dany@math.bgu.ac.il}

\author[I. Sabadini]{Irene Sabadini}
\address{(IS) Politecnico di
Milano\\Dipartimento di Matematica\\Via E. Bonardi, 9\\20133 Milano\\Italy}
\email{irene.sabadini@polimi.it}

\begin{document}
\maketitle
\begin{abstract}
We give a generalization of the Beurling-Lax theorem both in the complex and quaternionic settings. We consider in the first case
functions meromorphic in the right complex half-plane, and functions slice hypermeromorphic in the right quaternionic half-space in the second case. In both settings we also discuss a unified framework, which includes both the disk and the half-plane for the complex case and the
open unit ball and the half-space in the quaternionic setting.
\end{abstract}

\noindent AMS Classification: 47A56, 47B32, 30G35.

\noindent {\em Key words}: Beurling-Lax theorem, analytic functions in the unit disk, in the half-plane, slice hyperholomorphic functions in the quaternionic unit ball, in the quaternionic half space, de Branges Rovnyak spaces.
\date{today}
\tableofcontents
\section{Introduction}
\setcounter{equation}{0}

This paper mainly deals with a Beurling-Lax theorem for vector-valued functions meromorphic in the right open half-plane $\mathbb C_r$,
and slice hypermeromorphic in the right half-space $\mathbb H_+$ in the quaternionic setting. For $\alpha\in\mathbb C$ we denote by $R_\alpha$ the resolvent-like operator
\[
R_\alpha f(z)=\begin{cases}\,\, \dfrac{f(z)-f(\alpha)}{z-\alpha},\,\,\, z\not=\alpha,\\
\,\, f^\prime(\alpha),\,\,\,\,\,\,\,\quad\hspace{0.6cm} z=\alpha,\end{cases}
\]
where the (possibly vector-valued) function $f$ is analytic in a neighborhood of $\alpha$. The name comes from the resolvent identity
\begin{equation}
\label{resoid}
R_\alpha-R_\beta=(\alpha-\beta)R_\alpha R_\beta, \qquad \forall\alpha, \beta\in\mathbb C
\end{equation}
which they satisfy, and which we use in the sequel.\\

It is useful to remark that $R_0^*=M_z$ (the operator of multiplication by $z$) in the Hardy space $\mathbf H_2(\mathbb D)$ of the open unit disk $\mathbb D$ and that
\[
R_\alpha^*=(M_z+\overline{\alpha} I)^{-1},\quad \alpha\in\mathbb C_r,
\]
in the Hardy space $\mathbf H_2(\mathbb C_r)$ of the right half-plane (for the proof of this fact, one makes use of computations similar to those in the proof of Lemma \ref{Lemma27}).\smallskip

The Beurling-Lax theorem gives a characterization of the $M_z$-invariant subspaces of the Hardy space, see \cite{duren,MR92d:46066} and
\cite{laxphi} for the vector-valued case. In the scalar case they are spaces of the form $j\mathbf H_2(\mathbb D)$ in the disk case and $j\mathbf H_2(\mathbb C_r)$ in the half-plane case, where $j$ is an inner function (meaning that the operator $M_j$ of multiplication by $j$ is an isometry from the corresponding Hardy space into itself). The orthogonal complement of such a space is therefore $R_\alpha$-invariant (for appropriate choices of $\alpha$) and has reproducing kernel
\begin{equation}
\frac{1-j(z)\overline{j(w)}}{1-z\overline{w}}\,\,\text{(disk case)}\,\,\,\,\text{or}\,\,\,\,
\frac{1-j(z)\overline{j(w)}}{2\pi(z+\overline{w})}\,\,\text{(half-plane case)}.
\end{equation}
These functions are positive definite (in the open unit disk and the open right half-plane respectively) when $j$ is assumed analytic and contractive in $\mathbb D$ (respectively in $\mathbb C_r$), but not necessarily inner. Then
the corresponding reproducing kernel Hilbert space is not included isometrically, but only contractively, inside the underlying Hardy space.
\smallskip

One of the purposes of this work is to characterize reproducing kernel Hilbert  spaces with
reproducing kernel of the form $\dfrac{1-j(z)\overline{j(w)}}{2\pi(z+\overline{w})}$ for such $j$, and more generally Pontryagin spaces
when moreover $j$ is operator-valued.
\subsection{The case of Hardy spaces}
To put the study in perspective we review a few facts on Hardy spaces.
We begin by recalling the following result:
\begin{theorem}
The Hardy space $\mathbf H_2(\mathbb D)$ is the reproducing kernel Hilbert space with
reproducing kernel
\[
\frac{1}{1-z\overline{w}}.
\]
It is $R_\alpha$ invariant for $\alpha\in\mathbb D$, and the
following identity holds:
\begin{equation}
\label{equadb1}
\langle f,g\rangle+\alpha\langle R_\alpha f,g\rangle+\overline{\beta}\langle f,R_\beta g\rangle-(1-\alpha\overline{\beta})
\langle R_\alpha f,R_\beta g\rangle-\overline{g(\beta)}f(\alpha)=0,
\end{equation}
where $f,g\,\in\,\mathbf H_2(\mathbb D)$ and $\alpha,\beta\in\mathbb D$.
\end{theorem}

We will use \eqref{equadb1} in the proof of Proposition \ref{hannah} and for this reason we now give a quick proof of it. Note that in view of the resolvent
identity \eqref{resoid}, we have
\[
R_0(I+\alpha R_\alpha)=R_\alpha,\quad \alpha\in\mathbb D.
\]
The left hand-side of
\eqref{equadb1} may be rewritten as
\[
\langle (I+\alpha R_\alpha)f,(I+\beta R_\beta)g\rangle-\langle R_0(I+\alpha R_\alpha)f,R_0(I+\beta R_\beta)g\rangle=\overline{g(\beta)}
f(\alpha),
\]
or, setting $F=(I+\alpha R_\alpha)f$ and $G=(I+\beta R_\beta)g$,
\[
\langle F,G\rangle-\langle R_0 F,R_0G\rangle=\overline{G(0)}F(0),
\]
which is trivial in $\mathbf H_2(\mathbb D)$.
\\
In the case of the right half-plane $\mathbb C_r$ we have:
\begin{theorem}
The Hardy space $\mathbf H_2(\mathbb C_r)$ is the reproducing kernel Hilbert space with
reproducing kernel
\[
\frac{1}{2\pi(z+\overline{w})}.
\]
It is $R_\alpha$ invariant for $\alpha\in\mathbb C_r$, and the
following identity holds:
\begin{equation}
\label{equadb2}
\langle R_\alpha f,g\rangle+\langle f,R_\beta g\rangle+(\alpha+\overline{\beta})
\langle R_\alpha f,R_\beta g\rangle+ 2\pi \overline{g(\beta)}f(\alpha)=0,
\end{equation}
where $f,g\,\in\,\mathbf H_2(\mathbb C_r)$ and $\alpha,\beta\in\mathbb C_r$.
\end{theorem}
The proof that \eqref{equadb2} holds in $\mathbf H_2(\mathbb C_r)$ will be used in the sequel, and thus it will be recalled in Section \ref{subsec2.3}.\smallskip

It is worthwhile to mention that an approach to generalized Beurling-Lax theorems was developed by de Branges and Rovnyak, see \cite{dbbook,dbr1,dbr2}, and consists in leaving
the realm of the Hardy space, but keeping equalities \eqref{equadb1} or \eqref{equadb2} (or, some variations of these), and then work in the setting of reproducing kernel spaces; see \cite{ball-contrac,dbhsaf1,rov-66}. A related paper, which makes use of de Branges spaces and \eqref{equadb2} and uses Riccati equations to consider the case of singular Gram matrices, is \cite{MR2278223}.\\

In another approach, see \cite{dbr2}, one considers inequality in \eqref{equadb1},  setting $f=g$ and $\alpha=\beta=0$.
More generally, setting $f=g$ and $\alpha=\beta$ in equalities \eqref{equadb1} and \eqref{equadb2}, one can weaken the equalities to the
requirements
\begin{equation}
\label{inequadb1}
\langle f,f\rangle+\alpha\langle R_\alpha f,f\rangle+\overline{\alpha}\langle f,R_\alpha f\rangle-(1-|\alpha|^2)
\langle R_\alpha f,R_\alpha f\rangle-|f(\alpha)|^2\le 0,
\end{equation}
or
\begin{equation}
\label{inequadb2}
\langle R_\alpha f,f\rangle+\langle f,R_\alpha f\rangle+(2{\rm Re}\,\alpha)
\langle R_\alpha f,R_\alpha f\rangle+|f(\alpha)|^2\le 0.
\end{equation}
depending on the setting (disk or right half-plane). The corresponding Hilbert spaces are then contractively included inside the corresponding
Hardy spaces. This will not be true anymore when one introduces indefinite metrics.
\subsection{The case of Pontryagin spaces}
To motivate our results and provide the setting to the paper we recall some results from \cite{abds2-jfa,adrs}. The various notions related to
Pontryagin spaces are reviewed in Section \ref{sec21}, and a reader not familiar with the theory of indefinite inner product spaces can specialize the forthcoming discussion to the case of Hilbert spaces and positive definite functions.\smallskip

Let $\mathcal C$ and $\mathcal D$ be two
Pontryagin spaces with the same index of negativity. By $\mathbf L(\mathcal D,\mathcal C)$ we will denote the set of continuous linear operators from $\mathcal D$ to $\mathcal C$. A $\mathbf L(\mathcal D,\mathcal C)$-valued function $S$ analytic in some open subset $\Omega$ of the open unit disk is called a generalized Schur function if the $\mathbf L(\mathcal  C,\mathcal C)$-valued kernel
\begin{equation}
\label{michelle1234}
K_S(z,w)=\frac{I_{\mathcal C}-S(z)S(w)^*}{1-z\overline{w}},\quad z,w\in\Omega.
\end{equation}
has a finite number of negative squares, say $\kappa$, in $\Omega$.
Assuming $0\in\Omega$, it is proved in \cite{adrs} that $S$ is a generalized Schur function if and only if it can be written in the form
\begin{equation}
S(z)=D+zC(I_{\mathcal P}-zA)^{-1}B
\end{equation}
where $\mathcal P$ is a Pontryagin space with index of negativity $\kappa$ and where the operator-matrix
\begin{equation}
\label{s-adrs}
\begin{pmatrix}A&B\\ C&D\end{pmatrix}\,\,:\,\, \mathcal P\oplus \mathcal D\,\,\longrightarrow\,\,\mathcal P\oplus \mathcal C
\end{equation}
is coisometric. It follows that $S$ has a unique meromorphic extension to $\mathbb D$, and for $S$ so extended the kernel $K_S$ has still
$\kappa$ negative squares for $z,w$ in the domain of analyticity of $S$. We here recall that generalized Schur functions and related classes of vector-valued functions have been extensively studied by Krein and Langer; see for instance \cite{kl1,kl2,kl3}.\smallskip

The reproducing kernel Pontryagin space  $\mathcal P(S)$,
with reproducing kernel $K_S$ is $R_0$-invariant and the coisometry property in \eqref{s-adrs} implies that
\begin{equation}
[R_0 f,R_0f]_{\mathcal P}\le[f,f]_{\mathcal P}-[f(0),f(0)]_{\mathcal C},\quad \forall f\in\mathcal P.
\label{malpensa190715!!!!}
\end{equation}
Conversely,  the following characterization of $\mathcal P(S)$ spaces was given in \cite[Theorem 3.1.2, p. 85]{adrs}:
\begin{theorem}
Let $\mathcal C$ be a Pontryagin space, and let $\Omega$ be an open subset of the open unit disk $\mathbb D$ containing the origin.
Let $\mathcal P$ be a reproducing kernel Pontryagin space of $\mathcal C$-valued functions analytic in $\Omega$, which is $R_0$-invariant and such that \eqref{malpensa190715!!!!} holds in $\mathcal P$. Any element of $\mathcal P$ has a meromorphic extension to $\mathbb D$ and there exists a Pontryagin space $\mathcal C_1$ with ${\rm ind}_-(\mathcal C_1)={\rm ind}_-(\mathcal C)$ and a function $S\in\mathcal S_\kappa(\mathcal C_1,\mathcal C)$, with $\kappa={\rm ind}_-(\mathcal P)$, such that the reproducing kernel of the space $\mathcal P$ is of the form \eqref{michelle1234}.
\label{michelle0}
\end{theorem}

 Thus the function $S$ is meromorphic in $\mathbb D$, with domain of analyticity $\Omega(S)$. Formula
\eqref{michelle1234} means that elements of $\mathcal P$ are restrictions to $\Omega$ of the elements of the reproducing kernel
Hilbert space $\mathcal P(S)$ with reproducing kernel $K_S(z,w)$.\smallskip

A version of Theorem \ref{michelle0} in the quaternionic setting, and for slice hyperholomorphic functions, was proved in \cite[Theorem 7.1, p. 862]{MR3192300}.
The purpose of this note is to give a version of this result in the case of the right half-plane in the complex setting
(see Theorem \ref{michelle1}), and for the right half-space in the quaternionic setting (see Theorem \ref{michelle112345}).

Finally we remark that inequality \eqref{malpensa190715!!!!} can be set at an arbitrary point of the open unit disk $\mathbb D$ as
\begin{equation}
\label{radisk}
[R_\alpha f,R_\alpha f]_{\mathcal P}\le[(I_{\mathcal P}+\alpha R_\alpha)f,(I_{\mathcal P}+\alpha R_\alpha)f]_{\mathcal P}-
[f(\alpha),f(\alpha)]_{\mathcal C}.
\end{equation}
See \cite[(3.6) in Theorem 3.4]{ad3}. Furthermore, we will show that
it is possible to use the setting developed in \cite{MR1197502,ad-jfa,ad9} to write both Theorems \ref{michelle0} and \ref{michelle1} under a common setting.\\

The outline of the paper is as follows. It consists of three sections besides the introduction. The second section is devoted to the case of the right
half-plane case $\mathbb C_r$, and is divided into four subsections. In the first two subsections we review briefly some notions on Pontryagin spaces and on operator-valued generalized Schur functions associated to the right half-plane. We then prove the counterpart of Theorem \ref{michelle0}, and consider the particular case of spaces isometrically included in the Hardy space of the right half-plane. In the third section, divided into four subsections, we discuss the unified setting, to which we already alluded. We present the main aspects of this setting, discuss the Hardy space and the generalized Schur functions in this framework, and consider a general theorem, namely Theorem \ref{michelle00000}, which includes as particular cases Theorems \ref{michelle0} and \ref{michelle1}. In the fourth and last section we study the counterpart of Theorem \ref{michelle1} in the setting of slice hyperholomorphic functions in the half-space case and we discuss the unified setting. Since the composition of two slice hyperholomorphic functions is not, in general, slice hyperholomorphic, our results in the unified setting are proved using the subclass of quaternionic intrinsic functions.

\section{The complex-valued case}
\setcounter{equation}{0}
\subsection{Pontryagin spaces and their operators}
\label{sec21}
We begin this section by reviewing some basic facts on Pontryagin spaces. For more information we refer the reader to \cite{bognar,MR92m:47068,ikl}.
A vector space $\mathcal V$ endowed with an Hermitian form $[\cdot,\cdot]$ is called a Pontryagin space if it can be written as
\begin{equation}
\mathcal V=\mathcal V_+[+]\mathcal V_-,
\label{decomp1}
\end{equation}
where:\\
$(a)$ The spaces $(\mathcal V,[\cdot,\cdot])$ and $(\mathcal V_-,-[\cdot,\cdot])$ are Hilbert spaces, and $\mathcal V_-$ is finite dimensional.\\
$(b)$ $\mathcal V_+\cap\mathcal V_-=\left\{0\right\}$ and
\begin{equation}
[v_+,v_-]=0,\quad \forall v_+\in\mathcal V_+\quad\text{and}\quad v_-\in\mathcal V_-.
\label{orthogo}
\end{equation}
The decomposition \eqref{decomp1} is called a fundamental decomposition; it is not unique (unless $\mathcal V$ is a Hilbert space, or an anti Hilbert space), but all
spaces $\mathcal V_-$ appearing in a fundamental decompositions have the same dimension, called the index of negativity (or simply, the index) of the Pontryagin space. For a given fundamental decomposition, the map
\begin{equation}
\label{normpk}
\|v\|=[v_+,v_+]-[v_-,v_-]
\end{equation}
is a norm and $(\mathcal V,\|\cdot\|)$ is a Hilbert space. All norms \eqref{normpk} are equivalent, and define the topology of the
Pontryagin space. Continuity of linear operators is defined with respect to this topology.\smallskip

Given two Pontryagin spaces $(\mathcal V_1,[\cdot,\cdot]_1)$ and $(\mathcal V_2,[\cdot,\cdot]_2)$, the adjoint of a continuous linear operator from $\mathcal V_1$ into $\mathcal V_2$ is defined by
\begin{equation}
[Av_1,v_2]_2=[v_1,A^*v_2]_1,\quad \forall v_1\in\mathcal V_1\quad \text{and}\quad v_2\in\mathcal V_2.
\label{adjoint}
\end{equation}
The following important results in the theory of linear operators in Pontryagin spaces are used in the sequel,
see \cite[Theorem 1.4.2 (1), p. 29 and Theorem 1.3.5, p 26]{adrs} and the references therein:
\begin{theorem}
Let $(\mathcal V_1,[\cdot,\cdot]_1)$ and $(\mathcal V_2,[\cdot,\cdot]_2)$ be two Pontryagin spaces with same index of negativity, and let
$T$ be a densely defined contraction from $\mathcal D(T)\subset\mathcal V_1$ into $\mathcal V_2$, meaning that
\[
[Tv_1,Tv_1]_2\le[v_1,v_1]_1,\quad \forall v_1\in\mathcal D(T).
\]
Then, $T$ extends to an everywhere defined contraction, whose adjoint is also a contraction.
\end{theorem}

\begin{theorem} (see \cite{ikl}, \cite[Theorem 1.3.6 p. 26]{adrs}).
A contraction between two Pontryagin spaces with same index of negativity has a maximal strictly negative invariant subspace.
\end{theorem}
We conclude this section with the notion of negative squares.
\begin{definition}
Let $(\mathcal C,\,[\cdot,\cdot])$ be a Pontryagin space (the coefficient space). A $\mathbf L(
\mathcal C)$-valued function $K(z,w)$ defined for $z,w$ in a set $\Omega$ has a finite number, say $\kappa$, of negative squares in
$\Omega$ if it is Hermitian,
\[
K(z,w)=K(w,z)^*,\quad\forall z,w\in\Omega,
\]
and if the following condition holds: For every integer $N$ and every choice of $w_1,\ldots, w_N\in\Omega$ and $c_1,\ldots, c_N\in\mathcal C$, the $N\times N$ Hermitian matrix with $\ell k$ entry $[K(w_\ell,w_k)c_k,c_\ell]$ has at most $\kappa$ strictly negative eigenvalues and exactly $\kappa$ such eigenvalues for some choice of $N,w_1,\ldots, w_N,c_1,\ldots, c_N$.
\end{definition}

For the next result, which originates with the work of L. Schwartz \cite{schwartz}, see \cite[Theorems 1.1.2 and 1.1.3,p. 7]{adrs} and the references therein.
\begin{theorem}
There is a one-to-one correspondence between $\mathbf L(\mathcal C, \mathcal C)$-valued functions defined on $\Omega$ and having $\kappa$ negative squares there and reproducing kernel Pontryagin spaces  of functions defined on $\Omega$
with index of negativity $\kappa$.
\end{theorem}

\subsection{Generalized Schur functions}
\label{secqaz}
Let $\mathcal C$ and $\mathcal D$ be two
Pontryagin spaces with the same index of negativity.
A $\mathbf L(\mathcal D,\mathcal C)$-valued function $S$ analytic in some open subset $\Omega$ of the open right half-plane
$\mathbb C_r$ is called a generalized Schur function if the $\mathbf L(\mathcal  C)$-valued kernel
\begin{equation}
\label{michelle12345678}
K_S(z,w)=\frac{I_{\mathcal C}-S(z)S(w)^*}{2\pi(z+\overline{w})},\quad z,w\in\Omega.
\end{equation}
has a finite number of negative squares, say $\kappa$, in $\Omega$. Let $\alpha\in\Omega$. It follows from the analysis in \cite{abds2-jfa}
that a function $S$ is a generalized Schur function of the right half-plane if and only if it can be written in the form
\begin{equation}
\label{250715toParis}
S(z)=H+\frac{z-\alpha}{z+\overline{\alpha}}G\left(I_{\mathcal P}-\frac{z-\alpha}{z+\overline{\alpha}}T\right)^{-1}F,
\end{equation}
where $\mathcal P$ is a Pontryagin space with index of negativity $\kappa$ and where the operator-matrix
\begin{equation}
\label{s-adrs111}
\begin{pmatrix}T&F\\ G&H\end{pmatrix}\,\,:\,\, \mathcal P\oplus \mathcal D\,\,\longrightarrow\,\,\mathcal P\oplus \mathcal C
\end{equation}
is coisometric. It follows that $S$ has a unique meromorphic extension to $\mathbb C_r$, and for $S$ so extended the kernel $K_S$ has still
$\kappa$ negative squares for $z,w$ in the domain of analyticity of $S$. It follows that the space $\mathcal P(S)$ is $R_\alpha$-invariant and that
\begin{equation}
(2{\rm Re}\, \alpha)[R_\alpha f,R_\alpha f]_{\mathcal P(S)}+[R_\alpha f,f]_{\mathcal P(S)}+[f,R_\alpha f]_{\mathcal P(S)}+2\pi[f(\alpha),f(\alpha)]_{\mathcal C}\le 0
\label{malpensa190715}
\end{equation}
holds in $\mathcal P(S)$. Note that this inequality is the counterpart of \eqref{malpensa190715!!!!}, and appeared in \cite[(3.7) in Theorem 3.4]{ad3}.\\

We refer to  \cite{MR3337534} for other results on realizations of Schur functions in a half-plane.
\subsection{The structure theorem}
The main result of this section is the following theorem, which is the counterpart of Theorem \ref{michelle0} in the case of the right half-space.
To prove it, we follow closely the computations in \cite{adrs}. A key tool in the arguments is inequality \eqref{malpensa190715}
\begin{theorem}
\label{michelle1}
Let $\mathcal C$ be a Pontryagin space, and let $\Omega$ be an open subset of $\mathbb C_r$. Let $\alpha\in\Omega$ be fixed.
Let $\mathcal P$ be a reproducing kernel Pontryagin space of $\mathcal C$-valued functions analytic in $\Omega$, which is $R_\alpha$-invariant and such that \eqref{malpensa190715} holds in $\mathcal P$.
Then every function of $\mathcal P$ has a unique meromorphic extension to $\mathbb C_r$ and there exists a Pontryagin space $\mathcal C_1$ with ${\rm ind}_-(\mathcal C_1)={\rm ind}_-(\mathcal C)$ and a function $S\in\mathcal S_\kappa(\mathcal C_1,\mathcal C)$, with $\kappa={\rm ind}_-(\mathcal P)$, such that the reproducing kernel of $\mathcal P$ is of the form
\begin{equation}
\label{sarah123456}
K_S(z,w)=\frac{I_{\mathcal C}-S(z)S(w)^*}{2\pi(z+\overline{w})}.
\end{equation}
\end{theorem}

As for Theorem \ref{michelle0} we note that the function $S$ is meromorphic in $\mathbb C_r$, with domain of analyticity $\Omega(S)$. Formula
\eqref{sarah123456} means that elements of $\mathcal P$ are restrictions to $\Omega$ of the elements of the reproducing kernel
Hilbert space $\mathcal P(S)$ with reproducing kernel $K_S(z,w)$.\smallskip

\begin{proof}[Proof of Theorem \ref{michelle1}] We set
\[
k=(2{\rm Re}\, \alpha),\quad
T=kR_\alpha+I_{\mathcal P},\quad G=\sqrt{2\pi k}C_\alpha,\quad\text{and}\quad C=\begin{pmatrix}T\\ G\end{pmatrix},
\]
where $C_\alpha$ denotes the point evaluation at $\alpha$. After multiplying by $2\pi$, inequality \eqref{malpensa190715} may be rewritten as
\[
(kR_\alpha+I_{\mathcal P})^*(kR_\alpha+I_{\mathcal P})+2\pi kC_\alpha^*C_\alpha\le I_{\mathcal P},
\]
that is,
\begin{equation}
\label{bastille}
I_{\mathcal P}-C^*C\ge 0.
\end{equation}
By \cite[Theorem 1.3.4 (1), p. 25]{adrs} we have
\begin{equation}
\label{ineqfund}
{\rm ind}_-(I_{{\mathcal P} \oplus\mathcal C}-CC^*)+{\rm ind}_-(\mathcal P)={\rm ind}_-(I_{{\mathcal P}}-C^*C)+{\rm ind}_-(\mathcal P)+{\rm ind}_-({\mathcal C}).
\end{equation}
Using \eqref{bastille} we get:
\[
{\rm ind}_-(I_{{\mathcal P} \oplus\mathcal C}-CC^*)={\rm ind}_-({\mathcal C}).
\]
By the Bogn\'ar-Kr\'amli  theorem, see \cite[pp. 20-21]{adrs}, there exists a defect operator, that is there exists a
Pontryagin space $\mathcal C_1$ with same negative index as $\mathcal C$ and operators
\begin{equation}
\label{facto2}
F\in\mathbf L(\mathcal  C_1,\mathcal P)\quad\text{and}\quad H\in\mathbf L(\mathcal  C_1,\mathcal C)
\end{equation}
such that
\begin{equation}
\label{FH}
I_{{\mathcal P} \oplus\mathcal C}-CC^*=\begin{pmatrix}F\\ H\end{pmatrix}
\begin{pmatrix}F\\ H\end{pmatrix}^*.
\end{equation}
It follows that the operator matrix
\[
\begin{pmatrix}T&F\\ G&H\end{pmatrix}\,\,:\,\,{\mathcal P} \oplus\mathcal C_1\,\,\rightarrow\,\,{\mathcal P} \oplus\mathcal C
\]
is coisometric. We define $S$ via \eqref{250715toParis} in a neighborhood of the point $\alpha$. This formula defines a meromorphic function in $\mathbb C_r$, as is explained in \cite{adrs} for the disk case. When $\mathcal C$ is a Hilbert space, this follows from the fact that $T$ is then a contraction and has a maximal strictly negative invariant subspace. The case of Pontryagin coefficient spaces is reduced to the Hilbert space case using the Potapov-Ginzburg transform.\\

As in \cite[Theorem 2.1.2 (1), p. 44]{adrs} we have
\[
\dfrac{I_{\mathcal C}-S(z)S(w)^*}{1-\dfrac{z-\alpha}{z+\overline{\alpha}}\dfrac{\overline{w}-\overline{\alpha}}{\overline{w}+\alpha}}=
G\left(I_{\mathcal P}-\frac{z-\alpha}{z+\overline{\alpha}}T\right)^{-1}\left(I_{\mathcal P}-\frac{w-\alpha}{w+\overline{\alpha}}T\right)^{-*}G^*,
\]
which can be rewritten as
\begin{equation}
\label{macau}
(z+\overline{\alpha})(\overline{w}+\alpha)
\dfrac{I_{\mathcal C}-S(z)S(w)^*}{(z+\overline{w})(\alpha+\overline{\alpha})}=2\pi k
C_\alpha\left(I_{\mathcal P}-\frac{z-\alpha}{z+\overline{\alpha}}T\right)^{-1}\left(I_{\mathcal P}-\frac{w-\alpha}{w+\overline{\alpha}}T\right)^{-*}C_\alpha^*.
\end{equation}
To conclude the proof we show that the point evaluation $C_w$ is given by:
\begin{equation}\label{pointeval}
C_w=\frac{\alpha+\overline{\alpha}}{w+\overline{\alpha}}C_\alpha\left(I_{\mathcal P}-\frac{w-\alpha}{w+\overline{\alpha}}T\right)^{-1}.
\end{equation}
To this end, we note that
\begin{equation}
\label{mira}
(Tf)(z)=(kR_\alpha f)(z)+f(z)=\frac{(z+\overline{\alpha})f(z)-(\alpha+\overline{\alpha})f(\alpha)}{z-\alpha}.
\end{equation}
Let now $h\in\mathcal P$ and set
\[
\frac{\alpha+\overline{\alpha}}{w+\overline{\alpha}}\left(I_{\mathcal P}-\frac{w-\alpha}{w+\overline{\alpha}}T\right)^{-1}h=g.
\]
In view of \eqref{mira} we have
\[
\begin{split}
h(z)&=\frac{w+\overline{\alpha}}{\alpha+\overline{\alpha}}\left(g(z)-\frac{w-\alpha}{w+\overline{\alpha}}\frac{(z+\overline{\alpha})g(z)-(\alpha+\overline{\alpha})g(\alpha)}{z-\alpha}
\right).
\end{split}
\]
Setting $z=w$ in the above expression we obtain $h(w)=g(\alpha)$ so \eqref{pointeval} is proved. Thus \eqref{macau} can be rewritten as
\[
\dfrac{I_{\mathcal C}-S(z)S(w)^*}{2\pi(z+\overline{w})}=C_zC_w^*,
\]
and so $K_S(z,w)$ is the reproducing kernel of $\mathcal P$.
\end{proof}

\begin{remark}
{\rm
When $\mathcal C$ is separable the space $\mathcal P$ is also separable since it consists of analytic functions, and $\mathcal C_1$ can be chosen separable.}
\label{sepa}
\end{remark}
\subsection{The case of subspaces of the Hardy space}
\label{subsec2.3}
We now consider the special case where the space $\mathcal P$ (which we now denote by $\mathcal H$) is a Hilbert space isometrically included in the Hardy space $\mathbf H_2(\mathbb C_r)$ of the right half-plane. The following lemma is proved in \cite{ad3},
and implies that \eqref{malpensa190715} is in fact an equality in $\mathbf H_2(\mathbb C_r)$. Its proof is recalled for completeness and since it provides the ground to prove Lemma \ref{Lemma27q}.
\begin{lemma}\label{Lemma27}
Equality \eqref{equadb2} holds in $\mathbf H_2(\mathbb C_r)$.
\end{lemma}

\begin{proof}
We use the fact that $k(z,w)=\frac{1}{2\pi (z+\overline{w})}$ is the reproducing kernel of $\mathbf H_2(\mathbb C_r)$. We have for $\alpha,\beta,\mu,\nu$ in $\mathbb C_r$
\[
R_\alpha k(\cdot,\mu)=-\frac{1}{\alpha+\overline{\mu}} k(\cdot,\mu)\quad \text{and}\quad R_\alpha k(\cdot,\nu)=-\frac{1}{\alpha+\overline{\nu}} k(\cdot,\nu).
\]
It follows that
\[
\begin{split}
\langle R_\alpha k(\cdot,\mu), R_\beta k(\cdot,\nu)\rangle&=\frac{1}{(\alpha+\overline{\mu})(\overline{\beta}+\nu)}k(\nu,\mu)=
\frac{1}{2\pi(\alpha+\overline{\mu})(\overline{\beta}+\nu)(\nu+\overline{\mu})},\\
\langle R_\alpha k(\cdot,\mu), k(\cdot,\nu)\rangle&=-\frac{1}{\alpha+\overline{\mu}}k(\nu,\mu)=-\frac{1}{2\pi(\alpha+\overline{\mu})(\nu+\overline{\mu})},\\
\langle k(\cdot,\mu), R_\beta k(\cdot,\nu)\rangle&=-\frac{1}{\overline{\beta}+\nu}k(\nu,\mu)=-
\frac{1}{2\pi(\overline{\beta}+\nu)(\nu+\overline{\mu})}.
\end{split}
\]
Furthermore,
\[
2\pi\overline{k(\beta,\nu)}k(\alpha,\mu)=\frac{2\pi}{(2\pi)^2}\frac{1}{(\alpha+\overline{\mu})(\overline{\beta}+\nu)}=
\frac{1}{2\pi}\frac{1}{(\alpha+\overline{\mu})(\overline{\beta}+\nu)}.
\]
Equality \eqref{equadb2} for $f=k(\cdot,\mu)$ and $g=k(\cdot,\nu)$ is thus equivalent to
\[
-\frac{1}{(\alpha+\overline{\mu})(\nu+\overline{\mu})}
-
\frac{1}{(\overline{\beta}+\nu)(\nu+\overline{\mu})}+
\frac{\alpha+\overline{\beta}}{(\alpha+\overline{\mu})(\overline{\beta}+\nu)(\nu+\overline{\mu})}+
\frac{1}{(\alpha+\overline{\mu})(\overline{\beta}+\nu)}=0,
\]
which is clearly an identity. Note also that the linear span of the reproducing kernels form a dense subset of $\mathbf H_2(\mathbb C_r)$. To prove
\eqref{equadb2} for all $f,g\in\mathbf H_2(\mathbb C_r)$ we remark that $R_\alpha$ and $R_\beta$ are bounded there. This latter fact
can be proved in two different ways. First, using the fact that  $\mathbf H_2(\mathbb C_r)$ is isometrically included in the Lebesgue space
$\mathbf L_2(\mathbb R)$, writing
\[
R_\alpha f(z)=\frac{f(z)}{z-\alpha}-\frac{f(\alpha)}{z-\alpha},
\]
and using the Cauchy-Schwarz inequality. The second way (see \cite{ad3}) consists in remarking that \eqref{equadb2} implies, for $f$ in the linear span of the reproducing kernels, that
\[
(2{\rm Re}\,\alpha)\|R_\alpha f\|^2\le2\|R_\alpha f\|\|f\|+2\pi\|C_\alpha\|^2\|f\|^2.
\]
This inequality implies that, on a dense set
\[
\frac{\|R_\alpha f\|^2}{\|f\|^2}\le\frac{1}{{\rm Re}\,\alpha}\left(\frac{\|R_\alpha f\|}{\|f\|}+\pi\|C_\alpha\|^2\right),
\]
which in turn, implies that $R_\alpha$ extends to a bounded operator in $\mathbf H_2(\mathbb C_r)$.
\end{proof}

In the notation of Theorem \ref{michelle1} let us set $\mathcal C=\mathbb C$, and hence $\mathcal C_1$ is a Hilbert space. Thus there exists a function $S\in\mathcal S_0(\mathcal C_1,\mathbb C)$ such that $\mathcal H$ is the reproducing kernel
Hilbert space with reproducing kernel $\dfrac{1-S(z)S(w)^*}{2\pi(z+\overline{w})}$. For every $z\in\mathbb C$ the value $S(z)$ is a bounded
operator (in fact a contraction) from $\mathcal C_1$ into $\mathbb C$, which (since $\mathcal C_1$ is separable; see Remark \ref{sepa}) we will write in matrix form as
\[
\begin{pmatrix}s_1(z)&s_2(z)&\cdots\end{pmatrix},
\]
after choosing an orthonormal basis of $\mathcal C_1$. By the properties of a vector-valued analytic function
(see \cite{MR751959}) each of the functions $s_j$ is analytic. Since
\[
\frac{1}{2\pi(z+\overline{w})}=\frac{1-S(z)S(w)^*}{2\pi(z+\overline{w})}+\frac{S(z)S(w)^*}{2\pi(z+\overline{w})}
\]
and since $\mathcal H$ is isometrically included in $\mathbf H_2(\mathbb C_r)$, we have
\[
\left\langle \frac{1-S(z)S(w)^*}{2\pi(z+\overline{w})},\frac{S(z)S(v)^*}{2\pi(z+\overline{v})}\right\rangle_{\mathbf H_2(\mathbb C_r)}=0,\quad\forall
v,w\in\mathbb C_r.
\]
Using the reproducing kernel property (or Cauchy's theorem) we have
\[
\left\langle \frac{S(w)^*}{2\pi(z+\overline{w})}, \frac{S(v)^*}{2\pi(z+\overline{v})}\right\rangle_{\mathbf H_2(\mathbb C_r)\otimes\mathcal C_1}=
\left\langle \frac{S(z)S(w)^*}{2\pi(z+\overline{w})},\frac{S(z)S(v)^*}{2\pi(z+\overline{v})}\right\rangle_{\mathbf H_2(\mathbb C_r)},
\quad\forall v,w\in\mathbb C_r,
\]
where we denote by $\mathbf H_2(\mathbb C_r)\otimes\mathcal C_1$ the Hardy space of $\mathcal C_1$-valued functions. It follows that the operator of multiplication by $S$ is an isometry from the closed linear span in $\mathbf H_2(\mathbb C_r)\otimes\mathcal C_1$ of the functions
\[
z\,\mapsto\,\frac{S(w)^*}{2\pi(z+\overline{w})},\quad w\in\mathbb C_r,
\]
into $\mathbf H_2(\mathbb C_r)$. Let $j\in\mathbb N$ be such that $s_j\not\equiv 0$. Then the above isometry property implies that
the operator of multiplication by $s_j$ is an isometry from $\mathbf H_2(\mathbb C_r)$ into itself. By the arguments in the scalar setting
it follows that $s_j$ is inner and hence all the other $s_k$, $k\not=j$ are identically equal to $0$. Thus we can chose $\mathcal C_1=\mathbb
C$, and we obtain that $\mathcal H$ is the reproducing kernel Hilbert space with reproducing kernel $\dfrac{1-s_j(z)\overline{s_j(w)}}{2\pi
(z+\overline{w})}$, which means that $\mathcal H^\perp=s_j\mathbf H_2(\mathbb C_r)$. We thus get back to the scalar version of the
Beurling-Lax theorem.\\

More generally, consider a matrix $J\in\mathbb C^{n\times n}$, which is both self-adjoint and unitary:  $J=J^*=J^{-1}$ (such a matrix is called a signature matrix). Define $\mathbf H_2(\mathbb C_r,J)$ to be the space $\mathbf H_2(\mathbb C_r)^n$ endowed with the form
\[
[f,g]_J=\langle f,Jg\rangle_{\mathbf H_2(\mathbb C_r)^n},\quad f,g\in\mathbf H_2(\mathbb C_r)^n.
\]
As a side remark, note that $\mathbf H_2(\mathbb C_r,J)$ is a Krein space.
The identity \eqref{equadb2} holds in $\mathbf H_2(\mathbb C_r,J)$, when the coefficient space $\mathbb C^n$ is endowed with the form
\[
[c,d]_J=d^*Jc,\quad c,d\in\mathbb C^n.
\]
Theorem \ref{michelle1} gives then the characterization of spaces of the form $S\mathbf H_2(\mathbb C_r,J)$. See \cite{ballhelton,bh35} for related results.

\section{A unified setting in the complex case}
\setcounter{equation}{0}

In this section we show how we case use a unified setting to treat both the case of the unit disk and the case of the right half-plane.
We first briefly recall the setting developed in the papers \cite{ad-laa-herm,ad-jfa} and related papers.

\subsection{A unified setting}
The starting point is to consider an open connected subset $\Omega$ of the complex plane, and a pair of functions $a(z)$ and $b(z)$ analytic in $\Omega$ and such that the two sets
\[
\Omega_+=\left\{z\in\Omega\,\,;\,\,|b(z)|<|a(z)|\right\}\quad \text{and}\quad \Omega_-=\left\{z\in\Omega\,\,;\,\,|b(z)|>|a(z)|\right\}
\]
are both nonempty. Then the "boundary" set
\[
\Omega_0=\left\{z\in\Omega\,\,;\,\,|b(z)|=|a(z)|\right\}
\]
is also nonempty.

We set
\begin{equation}
\label{rho}
\rho(z,w)=a(z)\overline{a(w)}-b(z)\overline{b(w)}
\end{equation}
and
\[
\delta(z,w)=b(z)a(\alpha)-a(z)b(\alpha).
\]
The representation of $\rho$ is unique up to an hyperbolic translation, that is, if we have two different representations of $\rho$ as \eqref{rho},
\[
\rho(z,w)=a(z)\overline{a(w)}-b(z)\overline{b(w)}=c(z)\overline{c(w)}-d(z)\overline{d(w)},
\]
then
\[
\begin{pmatrix}a(z)&b(z)\end{pmatrix}=\begin{pmatrix}c(z)&d(z)\end{pmatrix}U
\]
where the matrix $U$ is such that
\[
UJ_0U^*=J_0,\quad\text{where}\quad J_0=\begin{pmatrix}1&0\\0&-1\end{pmatrix},
\]
and in particular the sets $\Omega_+,\Omega_-$ and $\Omega_0$ do not depend on the given representation of $\rho$.\smallskip

One introduces the resolvent-like operators
\[
(R(a,b,\alpha)f)(z)=\frac{a(z)f(z)-a(\alpha)f(\alpha)}{a(\alpha)b(z)-b(\alpha)a(z)}.
\]
We note (see \cite[equation (3.14), p. 9]{ad-jfa}) that
\begin{equation}
\label{rarb1}
a(\alpha)R(b,a,\alpha)+b(\alpha)R(a,b,\alpha)=-I.
\end{equation}

The case $a(z)=1$ and $b(z)=z$ corresponds to $\Omega_+=\mathbb D$, while the case $a(z)=\sqrt{2\pi}\frac{z+1}{2}$ and
$b(z)=\sqrt{2\pi}\frac{z-1}{2}$ corresponds to $\Omega_+=\mathbb C_r$.\\

Functions of the form \eqref{rho} seem to have been considered first in the papers \cite{lak} and \cite{nudel-93}. We also refer to \cite{adlv}
for a recent application to the Schur algorithm, and to \cite{ab7} for a sample application to interpolation.

\subsection{The Hardy space}

The function $\dfrac{1}{\rho(z,w)}=\dfrac{1}{a(z)\overline{a(w)}-b(z)\overline{b(w)}}$ is positive definite in $\Omega_+$. The associated reproducing kernel
Hilbert space will be denoted by $\mathbf H_2(\rho)$. Let $\sigma(z)=\dfrac{b(z)}{a(z)}$. For a given representation $\rho(z,w)=a(z)\overline{a(w)}-b(z)\overline{b(w)}$ of $\rho$ we have
\begin{equation}
\label{PMH2}
\mathbf H_2(\rho)=\left\{f(z)=\frac{F(\sigma(z))}{a(z)},\,\, F\in\mathbf H_2(\mathbb D)\right\}
\end{equation}
with norm  $\|f\|=\|F\|$.
\\
The following result is contained in \cite[Theorem 4.4]{ad-jfa}.
\begin{proposition}
\label{hannah}
The equality
\begin{equation}
\label{adjfa}
\begin{split}
\langle R(a,b,\alpha)f, R(a,b,\beta)g\rangle_{\mathbf H_2(\rho)}-\langle R(b,a,\alpha)f, R(b,a,\beta)g\rangle_{\mathbf H_2(\rho)}+\overline{g(\beta)}f(\alpha)=0,
\end{split}
\end{equation}
holds in $\mathbf H_2(\rho)$.
\end{proposition}

\begin{proof}
Let $f(z)=\frac{F(\sigma(z))}{a(z)}\in\mathbf H_2(\rho)$.
We first note the formulas
\begin{eqnarray}
\label{Rab1}
(R(a,b,\alpha)f)(z)&=&\frac{(R_{\sigma(\alpha)}F)(\sigma(z))}{a(\alpha)a(z)}\\
(R(b,a,\alpha)f)(z)&=&-\frac{(R_{\sigma(\alpha)}zF)(\sigma(z))}{a(\alpha)a(z)}.
\end{eqnarray}
Therefore, with $g(z)=\frac{G(\sigma(z))}{a(z)}$ another element in $\mathbf H_2(\rho)$, and using the formula
\[
(R_u(zf))(z)=f(z)+u(R_uf)(z),
\]
we have
\[
\begin{split}
\langle R(a,b,\alpha)f, R(a,b,\beta)g\rangle_{\mathbf H_2(\rho)}-
\langle R(b,a,\alpha)f, R(b,a,\beta)g\rangle_{\mathbf H_2(\rho)}+\overline{g(\beta)}f(\alpha)&=\\
&\hspace{-10cm}=\frac{1}{a(\alpha)\overline{a(\beta)}}\langle R_{\sigma(\alpha)}F,R_{\sigma(\beta)}G\rangle_{\mathbf H_2(\mathbb D)}-\\
&\hspace{-9.5cm}-
\frac{1}{a(\alpha)\overline{a(\beta)}}\langle F+\sigma(\alpha)R_{\sigma(\alpha)}F,G+\sigma(\beta)\ R_{\sigma(\beta)}G\rangle_{\mathbf H_2(\mathbb D)}+
\overline{g(\beta)}f(\alpha),
\end{split}
\]
that is,
\[
\begin{split}
\langle R_{\sigma(\alpha)}F, R_{\sigma(\beta)}G\rangle_{\mathbf H_2(\mathbb D)}-
\langle F+\sigma(\alpha)R_{\sigma(\alpha)}F,G+\sigma(\beta) R_{\sigma(\beta)}G\rangle_{\mathbf H_2(\mathbb D)}+
\overline{G(\sigma(\beta))}F(\sigma(\alpha)),
\end{split}
\]
which is equal to $0$ by\eqref{equadb1}.
\end{proof}

Both \eqref{equadb1} and \eqref{equadb2} are special cases of \eqref{adjfa} (see \cite[equation (4.2), p. 12]{ad-jfa} for the latter
equation). We will weaken this equality to the requirement
\begin{equation}
\label{ineq1234567}
\langle R(a,b,\alpha)f, R(a,b,\alpha)f\rangle-\langle R(b,a,\alpha)f, R(b,a,\alpha)f\rangle+|f(\alpha)|^2\le0.
\end{equation}

\subsection{Generalized Schur functions}

As in Section \ref{secqaz},  we consider $\mathcal C$ and $\mathcal D$ two
Pontryagin spaces with the same index of negativity.
A $\mathbf L(\mathcal D,\mathcal C)$-valued function $S$ analytic in some open subset of $\Omega_+\subset\Omega$
is called a generalized Schur function if the $\mathbf L(\mathcal  C,\mathcal C)$-valued kernel
\begin{equation}
\label{michelle12345678!!!!!!!!!!!!}
K_S(z,w)=\frac{I_{\mathcal C}-S(z)S(w)^*}{a(z)\overline{a(w)}-b(z)\overline{b(w)}},\quad z,w\in\Omega.
\end{equation}
has a finite number of negative squares, say $\kappa$, in $\Omega$.

\begin{proposition}\label{P32}
Using the above notation, let $\alpha\in\Omega_+$.
A function $S$ is a generalized Schur function if and only if it can be written in the form
\begin{equation}
\label{monique}
S(z)=H+\frac{\sigma(z)-\sigma(\alpha)}{1-\sigma(z)\overline{\sigma(\alpha)}}G\left(I_{\mathcal P}-\frac{\sigma(z)-\sigma(\alpha)}{1-\sigma(z)\overline{\sigma(\alpha)}}T\right)^{-1}F,
\end{equation}
where $\mathcal P$ is a Pontryagin space with index of negativity $\kappa$ and where the operator-matrix
\eqref{s-adrs111} is coisometric.
\end{proposition}

\begin{proof} By \cite[Proof of Theorem 2.4, p. 44]{abds2-jfa} (see in particular equations (2.17) and (2.18)) the function $S$ can be written
as $S(z)=M(\sigma(z))$, where $M$ is a  $\mathbf L(\mathcal D,\mathcal C)$-valued  generalized Schur function of the open unit disk
(analytic in, say, $\Omega_M\subset\mathbb D$), such that the kernel $K_M(z,w)=\frac{I-M(z)M(w)^*}{1-z\overline{w}}$ has the same number $\kappa$ of negative squares as $K_S(z,w)$ (for the proof of this fact see \cite[Theorem 1.1.4]{adrs}). Let
$u\in\mathbb D$ be a point in a neighborhood of which the function $M$ is analytic, and let $b_u(z)=\frac{z+u}{1+z\overline{u}}$. The function $M_u(z)=M(b_u(z))$ is a generalized Schur function of the open unit disk, analytic in a neighborhood of the origin. By
Theorem \ref{michelle0} it can be written as $M(b_u(z))=H+zG(I_{\mathcal P}-zT)^{-1}F$, where the space $\mathcal P$ and the operators
$T,F,G,H$ are as in \eqref{monique}. We can take $u=\sigma(\alpha)$ since $S(z)=M(\sigma(z)$ is analytic in an neighborhood of $\alpha$).
and replacing $z$ by $b_{-\sigma(\alpha)}(z)$ we have (since $b_u(b_{-u}(z))\equiv z$ for any $u\in\mathbb D$)
\[
M(z)=H+\frac{z-\sigma(\alpha)}{1-z\overline{\sigma(\alpha)}}G\left(I_{\mathcal P}-
\frac{z-\sigma(\alpha)}{1-z\overline{\sigma(\alpha)}}T\right)^{-1}F,
\]
The result follows by replacing $z$ by $\sigma(z)$.
\end{proof}

Still from \cite{abds2-jfa} and \cite{adrs} we have:

\begin{proposition}\label{P33}
Let $S$ be a  $\mathbf L(\mathcal D,\mathcal C)$-valued  generalized Schur function, analytic in some open subset of $\Omega_+\subset\Omega$, and with associated reproducing kernel Pontryagin space
$\mathcal P(S)$. Then, $\mathcal P(S)$ is $R(a,b,\alpha)$-invariant and the inequality
\begin{equation}
\label{ineq1234567!!!}
[R(a,b,\alpha)f, R(a,b,\alpha)f]_{\mathcal P(S)}-[R(b,a,\alpha)f, R(b,a,\alpha)f]_{\mathcal P(S)}+[f(\alpha),f(\alpha)]_{\mathcal C}\le0
\end{equation}
holds in it.
\end{proposition}
\begin{proof}
With $M$ as in the preceding proposition we have
\[
\frac{I-S(z)S(w)^*}{a(z)\overline{a(w)}-b(z)\overline{b(w)}}=\frac{1}{a(z)}\frac{I-M(\sigma(z))M(\sigma(w))^*}{1-\sigma(z)
\overline{\sigma(w)}}\frac{1}{\overline{a(w)}}.
\]
It follows that (compare with \eqref{PMH2})
\begin{equation}
\label{PSPM}
\mathcal P(S)=\left\{f(z)=\frac{F(\sigma(z))}{a(z)},\,\, F\in\mathcal P(M)\right\},
\end{equation}
with inner product defined by
\begin{equation}
\label{PSPM1}
[f,g]_{\mathcal P(S)}=[F,G]_{\mathcal P(M)},\,\,\,\text{with}\,\,\, g(z)=\frac{G(\sigma(z))}{a(z)},\,\, G\in\mathcal P(M).
\end{equation}
It follows that \eqref{Rab1} still holds. Taking into account \eqref{rarb1}, we have
\[
\begin{split}
(R(b,a,\alpha)f)(z)&=-\left(\frac{f+b(\alpha)R(a,b,\alpha)f}{a(\alpha)}\right)(z)\\
&=-\frac{\frac{F(\sigma(z))}{a(z)}+b(\alpha)\frac{(R_{\sigma(\alpha)}F)(\sigma(z))}{a(\alpha)a(z)}}{a(\alpha)}\\
&=-\frac{F(\sigma(z))+\sigma(\alpha)R_{\sigma(\alpha)}F(\sigma(z))}{a(z)a(\alpha)},
\end{split}
\]
and we see that \eqref{ineq1234567!!!} is equivalent to:
\begin{equation}
\begin{split}
\frac{1}{|a(\alpha)|^2}[R_{\sigma(\alpha)}F,R_{\sigma(\alpha)}F]_{\mathcal P(S)}-[\frac{F+\sigma(\alpha)R_{\sigma(\alpha)}F}{a(\alpha)},
\frac{F+\sigma(\alpha)R_{\sigma(\alpha)}F}{a(\alpha)}]_{\mathcal P(S)}+&\\
&\hspace{-8cm}+
\frac{1}{|a(\alpha)|^2}[F(\sigma(\alpha)),F(\sigma(\alpha))]_{\mathcal C}\le 0,\quad F\in\mathcal P(M).
\end{split}
\end{equation}
This last inequality is \eqref{radisk}. To conclude, it suffices to remark that \eqref{radisk} follows from \eqref{malpensa190715!!!!} in
$\mathcal P(M)$,
\begin{equation}
\label{Clothilde!!!}
[R_0 F,R_0F]_{\mathcal P(M)}\le[F,F]_{\mathcal P(M)}-[F(0),F(0)]_{\mathcal C},\quad \forall f\in\mathcal P(S),
\end{equation}
as is seen by setting in \eqref{Clothilde!!!} $(I+\sigma(\alpha)R_{\sigma(\alpha)})F$ instead of $F$.
\end{proof}
\subsection{The structure theorem}
The following theorem contains as special cases Theorems \ref{michelle0} and \ref{michelle1}.  Note that its proof relies on Theorem \ref{michelle0}.

\begin{theorem}
Let $\mathcal C$ be a Pontryagin space, and let $\Omega$ be an open subset of the open complex plane.
Let $\mathcal P$ be a reproducing kernel Pontryagin space of $\mathcal C$-valued functions analytic in $\Omega$, which is $R_0$-invariant and such that \eqref{ineq1234567!!!} holds in $\mathcal P$. Then every elements of $\mathcal P$ has a unique meromorphic extension to $\Omega_+$, and there exists a Pontryagin space $\mathcal C_1$ with ${\rm ind}_-(\mathcal C_1)={\rm ind}_-(\mathcal C)$ and a function $S\in\mathcal S_\kappa(\mathcal C_1,\mathcal C)$, with $\kappa={\rm ind}_-(\mathcal P)$, such that the reproducing kernel of the space $\mathcal P$ is of the form \eqref{michelle12345678!!!!!!!!!!!!}.
\label{michelle00000}
\end{theorem}

\begin{proof} We proceed in a number of steps.\\

STEP 1: {\sl \eqref{ineq1234567!!!} can be rewritten as
\begin{equation}
\label{newenq}
\begin{split}
\left(\frac{(|a(\alpha)|^2-|b(\alpha)|^2)R(a,b,\alpha)-\overline{b(\alpha)}I_{\mathcal P}}{\overline{a(\alpha)}}\right)^*\left(\frac{(|a(\alpha)|^2-|b(\alpha)|^2)R(a,b,\alpha)-\overline{b(\alpha)}I_{\mathcal P}}{\overline{a(\alpha)}}\right)+\\
&\hspace{-5.5cm}+(|a(\alpha)|^2-|b(\alpha)|^2)C_\alpha^*C_\alpha\le I_{\mathcal P}.
\end{split}
\end{equation}
}

Recall that $a(\alpha)\not=0$ since $\alpha\in\Omega_+$.
Taking into account \eqref{rarb1} we rewrite \eqref{ineq1234567!!!} as
\[
R(a,b,\alpha)^*R(a,b,\alpha)-\frac{1}{|a(\alpha)|^2}(I_{\mathcal P}+b(\alpha)R(a,b,\alpha))^*(I_{\mathcal P}+b(\alpha)R(a,b,\alpha))
+C_\alpha^* C_\alpha\le 0,
\]
that is, \eqref{newenq}.\\

We now proceed as in the proof of Theorem \ref{michelle1} with now
\[
T=\frac{(|a(\alpha)|^2-|b(\alpha)|^2)R(a,b,\alpha)-\overline{b(\alpha)}I_{\mathcal P}}{\overline{a(\alpha)}},\quad G=\sqrt{(|a(\alpha)|^2-|b(\alpha)|^2)}
C_\alpha,\,\,\text{and}\,\, C=\begin{pmatrix}T\\ G\end{pmatrix}.
\]
We define operators $F,H$ via \eqref{FH} and set $\sigma(z)=\dfrac{b(z)}{a(z)}$.\\

STEP 2: {\it It holds that
\begin{eqnarray}
\label{late123456}
(1-\sigma(z)\overline{\sigma(\alpha)})I_{\mathcal P}-(\sigma(z)-\sigma(\alpha))T&=&
\frac{\rho(\alpha,\alpha)}{|a(\alpha)|^2}\left(I_{\mathcal P}-\frac{\delta(z,\alpha)}{\alpha(z)}R(a,b,\alpha)\right)\\
a(\alpha)\left(R(b,a,\alpha)+\frac{b(z)}{a(z)}R(a,b,\alpha)\right)&=&-I_{\mathcal P}+\frac{\delta(z,\alpha)}{a(z)}R(a,b,\alpha).
\label{late123456789}
\end{eqnarray}}
Indeed, we have
\[
\begin{split}
(1-\sigma(z)\overline{\sigma(\alpha)})I_{\mathcal P}-(\sigma(z)-\sigma(\alpha))T&=\\
&\hspace{-5.5cm}=\left(1-\frac{b(z)\overline{b(\alpha)}}{a(z)
\overline{a(\alpha)}}\right)I_{\mathcal P}-\left(\frac{b(z)}{a(z)}-\frac{b(\alpha)}{a(\alpha)}\right)\left(
\frac{(|a(\alpha)|^2-|b(\alpha)|^2)R(a,b,\alpha)-\overline{b(\alpha)}I_{\mathcal P}}{\overline{a(\alpha)}}\right)\\
&\hspace{-5.5cm}=\frac{\rho(\alpha,\alpha)}{|a(\alpha)|^2}\left(I_{\mathcal P}-\frac{\delta(z,\alpha)}{a(z)}R(a,b,\alpha)\right),
\end{split}
\]
which proves \eqref{late123456}. On the other hand,
\[
\begin{split}
a(\alpha)\left(R(b,a,\alpha)+\frac{b(z)}{a(z)}R(a,b,\alpha)\right)&=a(\alpha)R(b,a,\alpha)+a(\alpha)\frac{b(z)}{a(z)}R(a,b,\alpha)\\
&=-I_{\mathcal P}-b(\alpha)R(a,b,\alpha)+a(\alpha)\frac{b(z)}{a(z)}R(a,b,\alpha)\\
&=-I_{\mathcal P}+\frac{\delta(z,\alpha)}{a(z)}R(a,b,\alpha),
\end{split}
\]
which is \eqref{late123456789}.\\

STEP 3: {\sl Let $S$ be given by \eqref{monique}. Then,
\begin{equation}
\label{Sab}
S(z)=H-\frac{|a(\alpha)|^2(b(z)a(\alpha)-a(z)b(\alpha))}{a(\alpha)^2\rho(\alpha,\alpha)}G(a(z)R(b,a,\alpha)+b(z)R(a,b,\alpha))^{-1}F,
\end{equation}
and
\begin{equation}
\label{kernab}
\begin{split}
\frac{I_{\mathcal C}-S(z)S(w)^*}{a(z)\overline{a(w)}-b(z)\overline{b(w)}}&=\\
&\hspace{-2.5cm}=
C_\alpha\left(a(z)R(b,a,\alpha)+b(z)R(a,b,\alpha)\right)^{-1}
\left(a(w)R(b,a,\alpha)+b(w)R(a,b,\alpha)\right)^{-*}C_\alpha^*.
\end{split}
\end{equation}
}
It follows from \eqref{late123456}-\eqref{late123456789} that
\[
(1-\sigma(z)\overline{\sigma(\alpha)})I_{\mathcal P}-(\sigma(z)-\sigma(\alpha))T=-\frac{\rho(\alpha,\alpha)a(\alpha)}{|a(\alpha)|^2a(z)}
\left(a(z)R(b,a,\alpha)+b(z)R(a,b,\alpha)\right).
\]
We now plug this expression in \eqref{monique}, taking into account that
\[
\sigma(z)-\sigma(\alpha)=\frac{b(z)a(\alpha)-a(z)b(\alpha)}{a(z)a(\alpha)}
\]
to get \eqref{Sab}.\smallskip

We now prove \eqref{kernab}. From the similar formula for $a(z)=1$ and $b(z)=z$ (see \cite{adrs}) we have, with
\[
b_\sigma(z)=\frac{\sigma(z)-\sigma(\alpha)}{1-\sigma(z)\overline{\sigma(\alpha)}},
\]
\[
\frac{I_{\mathcal C}-S(z)S(w)^*}{1-b_\sigma(z)\overline{b_\sigma(w)}}=G(I_{\mathcal P}-b_\sigma(z)T)^{-1}
(I_{\mathcal P}-b_\sigma(w)T)^{-*}G^*,
\]
that is,
\[
\begin{split}
(1-\sigma(z)\overline{\sigma(\alpha)})\frac{I_{\mathcal C}-S(z)S(w)^*}{(1-|\sigma(\alpha)|^2)(1-\sigma(z)\overline{\sigma(w)})}
(1-\sigma(\alpha)\overline{\sigma(w)})&=\\
&\hspace{-10cm}=
(1-\sigma(z)\overline{\sigma(\alpha)})G
\left((1-\sigma(z)\overline{\sigma(\alpha)})I_{\mathcal P}-(\sigma(z)-\sigma(\alpha))T\right)^{-1}\times\\
&\hspace{-9.5cm}\times
\left((1-\sigma(w)\overline{\sigma(\alpha)})I_{\mathcal P}-(\sigma(w)-\sigma(\alpha))T\right)^{-*}G^*(1-\sigma(\alpha)\overline{\sigma(w)}).
\end{split}
\]
Using the definition of $G$ and \eqref{late123456}-\eqref{late123456789} this last equality is equivalent to:
\[
\begin{split}
\frac{I_{\mathcal C}-S(z)S(w)^*}{(1-|\sigma(\alpha)|^2)(1-\sigma(z)\overline{\sigma(w)})}&=\\
&\hspace{-3cm}=
\frac{|a(\alpha)|^2a(z)}{a(\alpha)\rho(\alpha,\alpha)}
C_\alpha\left(a(z)R(b,a,\alpha)+b(z)R(a,b,\alpha)\right)^{-1}\times\\
&\hspace{-2.5cm}\times
\left(a(w)R(b,a,\alpha)+b(w)R(a,b,\alpha)\right)^{-*}C_\alpha^*\frac{|a(\alpha)|^2\overline{a(w)}}{\overline{a(\alpha)}\rho(\alpha,\alpha)}\rho(\alpha,\alpha),
\end{split}
\]
from which the result follows since
\[
\frac{a(z)\overline{a(w)}}{1-|\rho(\alpha)|^2}=\frac{|a(\alpha)|^2a(z)}{a(\alpha)\rho(\alpha,\alpha)}\frac{|a(\alpha)|^2\overline{a(w)}}{\overline{a(\alpha)}\rho(\alpha,\alpha)}\rho(\alpha,\alpha).
\]
STEP 4: {\sl With $S$ as in the previous step, the reproducing kernel of $\mathcal P$ is equal to
\[
\frac{I_{\mathcal C}-S(z)S(w)^*}{a(z)\overline{a(w)}-b(z)\overline{b(w)}}.
\]
}

We verify that
\begin{equation}
\label{cwcz}
C_w=-C_\alpha\left(a(w)R(b,a,\alpha)+b(w)R(a,b,\alpha)\right)^{-1}
\end{equation}
(note that \eqref{cwcz} holds for $w=\alpha$ in view of \eqref{rarb1}).
We write
\[
-\left(a(w)R(b,a,\alpha)+b(w)R(a,b,\alpha)\right)^{-1}h=g.
\]
Then
\[
\begin{split}
-h(z)&=\left(\left(a(w)R(b,a,\alpha)+b(w)R(a,b,\alpha)\right)g\right)(z)\\
&=\frac{a(w)b(z)-b(w)a(z)}{b(\alpha)a(z)-a(\alpha)b(z)}g(z)+\frac{a(w)b(\alpha)-b(w)a(\alpha)}{a(\alpha)b(z)-b(\alpha)a(z)}g(\alpha).
\end{split}
\]
Thus $-h(w)=g(\alpha)$ and \eqref{cwcz} is proved. Thus the reproducing kernel can be written as $C_zC_w^*$,  which ends the proof.
\end{proof}

\begin{remark}{\rm Since \eqref{adjfa} holds in $\mathbf H_2(\rho)$ we obtain that the orthogonal complement of the space $\mathcal P$ is of the form
$S\mathbf H_2(\rho)$.
}
\end{remark}
\section{The quaternionic-valued case}
\setcounter{equation}{0}

In this part we consider the case of quaternionic-valued slice hyperholomorphic functions. The results for quaternionic Pontryagin spaces corresponding to the results in Section \ref{sec21} can be found in \cite{acs_book,acs1,MR3192300,acs_OT_244}, and are not repeated, but we provide precise references. The counterpart of Theorem \ref{michelle0} and of Theorem \ref{michelle1} have been proved in \cite[Theorem 7.1, p. 862]{MR3192300} and in \cite[Theorem 7.2, p. 122]{acls_milan}, respectively. Here we mainly consider the counterpart of Theorem \ref{michelle1} for the half-space case. The Beurling-Lax theorem for slice hyperholomorphic functions on the open unit ball of the quaternions is discussed in Section 8.4 of \cite{acs_book}.
\\
To keep the exposition self-contained, in this section we also provide the necessary background on slice hyperholomorphic functions.
We begin by providing some basic fact about slice hyperholomorphic functions with values in a Banach or Hilbert space. For more information, we refer the reader to \cite{acs_book}.

\setcounter{equation}{0}

\subsection{Slice hyperholomorphic functions}

 We will denote by $\mathbb H$ the skew field of quaternions. It contains elements of the form $p=x_0+x_1i+x_2j+x_3k$ where $x_\ell\in\mathbb R$, and $i,j,k$ are imaginary units such that $i^2=j^2=-1$, $ij=-ji$ and $k=ij$. The conjugate of $p$ is denoted by $\bar p$ and $\bar p=x_0-x_1i-x_2j-x_3k$. Note that $p\bar p=\bar p p=|p|^2=x_0^2+x_1^2+x_2^2+x_3^2$.
 \\
The set $\mathbb S$ of quaternions $p$ such that $p^2=-1$ consists of purely imaginary quaternions, namely quaternions of the form $p=x_1i+x_2j+x_3k$, with $|p|=1$. It is a $2$-dimensional sphere in $\mathbb H$ identified with the Euclidean spaces $\mathbb R^4$.\\
Let $I\in\mathbb S$; then the set of elements of the form $x+Iy$ is a complex plane denoted by $\mathbb C_I$. Every nonreal quaternion $p$ belongs to a unique complex plane $\mathbb C_I$ where $I$ is determined by its imaginary part, normalized.\\
By $\mathbb H_+$ we denote the open half-space
\[
\mathbb H_+=\left\{p\in\mathbb H\,\,;{\rm Re}\, p>0\right\},
\]
which intersects the positive real axis.

The counterpart of Schur functions in the slice hyperholomorphic setting were introduced in \cite{acls_milan} and further studied in \cite{acs_book} to which we refer the reader for more details. Here we give the following definition of  slice hyperholomorphic functions (equivalent to the one given in \cite{acls_milan}):
\begin{definition}
Given
  be a two sided quaternionic Banach (or Hilbert) space $\mathcal X$,
a real differentiable function $f: \Omega\subseteq\mathbb H\to \mathcal{X}$ is (weakly)
slice hyperholomorphic if and only if $\frac 12(\partial_x
+I\partial_y)f_I(x+Iy)=0$ for all $I\in\mathbb{S}$.
\end{definition}
\begin{remark}{\rm
If, under the same hypothesis, one imposes $\frac 12\partial_x f_I(x+Iy)+ \frac 12
\partial_yf_I(x+Iy)I=0$ for all $I\in\mathbb{S}$ the function $f$ is said to be right (weakly)
slice hyperholomorphic.\\
In particular, when a function $f$ defined on $\Omega$ is quaternionic valued, we say that it is slice hyperholomorphic if
and only if $\frac 12(\partial_x
+I\partial_y)f_I(x+Iy)=0$ for all $I\in\mathbb{S}$.
}
\end{remark}

\begin{remark}{\rm
Given  a two-sided
quaternionic Hilbert space $\mathcal X$ and a $\mathcal X$-valued function $f$ slice
hyperholomorphic in a neighborhood of $\alpha\in\mathbb R$, then $f$ can be written as
a convergent power series
\[
f(p)=\sum_{n=0}^\infty (p-\alpha)^nf_n,
\]
where the coefficients $f_n\in\mathcal X$.}
\end{remark}

In the sequel, we will consider open sets $\Omega$ which are axially symmetric slice domains (in short, s-domains). \begin{definition} Let $\Omega\subseteq \mathbb H$. We say that $\Omega$ is
axially symmetric if whenever $p=x_0+Iy_0$ belongs to $\Omega$ also all the elements of the form $x_0+Jy_0$, $J\in\mathbb S$ belongs to $\Omega$. \\
$\Omega$ is said to be a slice domain if it is a connected open set whose intersection with any complex plane $\mathbb C_I$ is connected.
\end{definition}
Given $p=x_0+Iy_0$ the set of elements of the form $x_0+Jy_0$, $J\in\mathbb S$ is a $2$-dimensional sphere denoted by $[p]$. The sphere $[p]$ contains elements of the form $q^{-1} p q$ for $q\not=0$.
\begin{remark}{\rm
The Identity Principle, see \cite{acls_milan,acs_book}, implies that two slice hyperholomorphic functions defined on an s-domain and $\mathcal X$-valued coincide if their restrictions to the real axis coincide. Moreover, any real analytic function $f: [a,b]\subseteq\mathbb R \to \mathcal X$ can be extended to a function, denoted by ${\rm ext}(f)$, which is slice hyperholomorphic on a suitable axially symmetric s-domain $\Omega$ containing $[a,b]$. In fact, for any $x_0\in [a,b]$ the function $f$ can be written as $f(x)=\sum_{n\geq 0} x^n f_n$, $f_n\in\mathcal X$,  for $x$ such that $|x-x_0|<\varepsilon$ and thus $({\rm ext}f)(p)= \sum_{n\geq 0} p^n f_n$ converges and defines a slice hyperholomorphic function for $|p-x_0|< \varepsilon_{x_0}$. Thus we can set $B(x_0,\varepsilon_{x_0})=\{p\in\mathbb H\ :\ |p-x_0|< \varepsilon_{x_0}\}$ and  $\Omega=\cup_{x_0\in [a,b]} B(x_0,\varepsilon_{x_0})$.}
\label{extension!!!!}
\end{remark}
The pointwise multiplication of two slice hyperholomorphic functions is not, in general hyperholomorphic, so we introduce the following notion of multiplication:
\begin{definition}
Let $\Omega\subseteq\mathbb H$ be an axially symmetric s-domain and let
$f,g:  \Omega\to \mathcal{X}$ be slice hyperholomorphic functions with values in a two sided quaternionic Banach algebra $\mathcal X$. Let
$f(x+Iy)=\alpha(x,y)+I\beta(x,y)$,
$g(x+Iy)=\gamma(x,y)+I\delta(x,y)$. Then we define
\begin{equation}\label{starleft}
(f\star g)(x+Iy):= (\alpha\gamma -\beta \delta)(x,y)+
I(\alpha\delta +\beta \gamma )(x,y).
\end{equation}
\end{definition}
It can be verified that $f\star g$ is slice
hyperholomorphic. In a similar manner, one can define a multiplication, denoted by $\star_r$, between right slice hyperholomorphic functions.
\begin{remark}{\rm In particular, let $f:\,\rho_S(A)\cap \mathbb R\to \mathcal X$ be the function $f(x)=(I-xA)^{-1}$, where $\rho_S(A)$ denotes the S-resolvent of $A$.
Then
\[
p^{-1}S_R^{-1}(p^{-1},A)=(I-\bar p A)(I-2 {\rm Re}(p) A
+|p|^2 A^2)^{-1}
\]
is the unique slice hyperholomorphic extension to $\rho_S(A)$. This extension will denoted by $(I-pA)^{-\star}$, in fact it is the $\star$-inverse of $(I-pA)$.}
\end{remark}

In the sequel, we will make use of the following result, see
\cite{acls_milan,acs_book}:
\begin{proposition}
Let $A$ be a bounded linear operator from a right-sided
quaternionic Banach $\mathcal X$ space into itself, and let $G$
be a bounded linear operator from $\mathcal X$ into $\mathcal Y$,
where $\mathcal Y$ is a  two sided quaternionic Banach space. The
slice hyperholomorphic extension of $G(I-xA)^{-1}$,
$1/x\in\sigma_S(A)\cap\mathbb{R}$, is
\[
(G-\overline{p}GA)(I-2{\rm Re}(p)\, A +|p|^2A^2)^{-1}.
\]
\label{formula060813}
\end{proposition}
With an abuse of
notation, we will write  $G\star (I-pA)^{-\star}$ meaning the
expression $(G-\overline{p}GA)(I-2{\rm Re}(p)\, A
+|p|^2A^2)^{-1}$.

\begin{remark}{\rm
The composition $f\circ g$ of two slice hyperholomorphic functions is not, in general, slice hyperholomorphic unless additional hypothesis are assumed. We say that a function slice hyperholomorphic on $\Omega$ is {\em quaternionic intrinsic}
if it is quaternionic valued and, for every $I\in\mathbb S$, it takes elements belonging to $\Omega\cap\mathbb C_I$ to $\mathbb C_I$. The composition of two slice hyperholomorphic functions $f\circ g$, when defined, is slice hyperholomorphic when $g$ is quaternionic intrinsic.\\
In particular, the composition with the quaternionic counterpart of the operator $R_\alpha$ will not be hyperholomorphic, unless $\alpha\in\mathbb R$. Note also that if $f$ is quaternionic intrinsic and $g$ is slice hyperholomorphic, then $f\star g=fg$ and $f^{-\star}=f^{-1}$.}
\label{rk:intrinsic}
\end{remark}
 In this setting, $R_\alpha$ is defined as
\begin{equation}
R_{\alpha}f(p)=(p-\alpha)^{-1}(f(p)-f(\alpha))\stackrel{\rm def.}{=}
\begin{cases}\,\,\sum_{n=1}^\infty (p-\alpha)^{n-1}f_n,\quad
p\not=\alpha,\\
\quad f_1,\hspace{2.8cm}\,\,\, p=\alpha,
\end{cases}
\label{RX0}
\end{equation}
where $f(p)=\sum_{n=0}^\infty (p-\alpha)^{n}f_n$.
\\
We end this part by recalling the notion of slice hypermeromorphic functions:
\begin{definition}
Let $\mathcal X$ be a two-sided quaternionic Banach space.
We say that a function $f:\, \Omega\to \mathcal X$ is (weakly) slice
hypermeromorphic if for any $\Lambda$ in the dual of $\mathcal X$, the function $\Lambda f: \, \Omega\to
\mathbb H$ is slice hypermeromorphic in $\Omega$.
\end{definition}
Note that
the previous definition means that $\Lambda f$ is slice hyperholomorphic in an open set $\Omega'$, where the points belonging to $\Omega\setminus \Omega'$ are the poles of $\Lambda f$ and $(\Omega\setminus \Omega')\cap\mathbb C_I$ has no point limit in $\Omega\cap\mathbb C_I$ for $I\in\mathbb S$.

\subsection{The Hardy space of the open half-space $\mathbb H_+$}
In this subsection we recall the definition of the Hardy space of the half space $\mathbb H_+$.
\begin{definition}
We define $\mathbf{H}_2(\mathbb H_+)$ as the space of slice
hyperholomorphic functions on $\mathbb H_+$ such that
\begin{equation}\label{xge0integ}
\sup_{I\in\mathbb S}\int_{-\infty}^{+\infty} | f(Iy) |^2 dy <\infty .
\end{equation}
\end{definition}
Let us consider the function
\begin{equation}\label{kernel}
k(p,q)=(\bar p +\bar q)(|p|^2 +2{\rm Re}(p) \bar q +\bar q^2)^{-1} =(|q|^2 +2{\rm Re}(q) p + p^2)^{-1}( p + q)
\end{equation}
which is slice hyperholomorphic in $p$ and $\bar q$ on the left and on
the right, respectively in its domain of definition. Note that we can write
\[
k(p,q)=(p+ \bar q)^{-\star}
\]
where the $\star$-inverse is computed with respect to $p$.\\
We have:
\begin{proposition}
The kernel $\frac{1}{2\pi}k(p,q)$ is reproducing, i.e. for any
$f\in\mathbf{H}_2(\mathbb H_+)$
\[
f(p)=\int_{-\infty}^\infty \frac{1}{2\pi}k(p,Iy) f(Iy) dy.
\]
\end{proposition}

The $\mathbf
L(\mathcal D,\mathcal C)$-valued function $S$ slice
hypermeromorphic in an axially symmetric s-domain $\Omega$ which
intersects the positive real line belongs to the class $\mathcal
S_\kappa(\Omega)$ if the kernel
\[
K_S(p,q)=k(p,q) I_{\mathcal C}-S(p)\star k(p,q)I_{\mathcal C}\star_rS(q)^*
\]
has $\kappa$ negative squares in $\Omega$, where $k(p,q)$ is defined in (\ref{kernel}).

\subsection{Generalized Schur functions}
In this section we discuss the quaternionic counterpart of Theorem \ref{michelle1}, see \cite{acls_milan}.  Let $\mathcal C$ and $\mathcal D$ be a pair of two-sided quaternionic Pontryagin spaces with the same index of negativity.
The $\mathbf L(\mathcal D,\mathcal C)$-valued function $S$ analytic in some open, axially symmetric, subset $\Omega$ of the open right half-space
$\mathbb H_+$ is called a generalized Schur function if the $\mathbf L(\mathcal  C,\mathcal C)$-valued kernel $K_S(p,q)$ solution of the
equation
\begin{equation}
\label{michelle12345678910}
2\pi(pK_S(p,q)+K_S(p,q)^*\bar q)=I_{\mathcal C}-S(p)S(q)^*,\quad p,q\in\Omega
\end{equation}
has a finite number of negative squares, say $\kappa$, in $\Omega$. Let $\alpha\in\Omega\cap[0,\infty)$. A function $S$ is a generalized Schur function of the right half-plane if and only if it can be written in the form
\begin{equation}
\label{250715toParis!!!!}
S(p)=H-\frac{p-\alpha}{p+{\alpha}}G\star\left(I_{\mathcal P}-\frac{p-\alpha}{p+{\alpha}}T\right)^{-\star}F.
\end{equation}
where $\mathcal P$ is a right-sided quaternionic Pontryagin space with index of negativity $\kappa$ and where the operator-matrix
\begin{equation}
\label{s-adrs-123}
\begin{pmatrix}T&F\\ G&H\end{pmatrix}\,\,:\,\, \mathcal P\oplus \mathcal D\,\,\longrightarrow\,\,\mathcal P\oplus \mathcal C
\end{equation}
is coisometric. It follows that $S$ has a unique slice hyperholomorphic extension to $\mathbb H_+$, and for $S$ so extended the kernel $K_S$ has still
$\kappa$ negative squares for $p,q$ in the domain of slice hyperholomorphicity of $S$. It follows that the space $\mathcal P(S)$ is $R_\alpha$-invariant.\smallskip

We note that equation \eqref{250715toParis!!!!} gives the (unique) slice hypermeromorphic extension of
\[
S(x)=H-\frac{x-\alpha}{x+{\alpha}}G\left(I_{\mathcal P}-\frac{x-\alpha}{x+{\alpha}}T\right)^{-1}F
\]
from a real neighborhood $(\alpha-\eta,\alpha+\eta)$ to the open right half-space.
\\
We also note that the quaternionic analog of Theorem \ref{michelle0} proved in \cite[Theorem 7.1, p. 862]{MR3192300} assumes the inequality
\begin{equation}
[R_0 f,R_0f]_{\mathcal P}\le[f,f]_{\mathcal P}-[f(0),f(0)]_{\mathcal C},\quad \forall f\in\mathcal P.
\label{malpensa190715!!!!q}
\end{equation}
It is immediate that $R_0(I+\alpha R_{\alpha})=R_{\alpha}$, $\alpha\in\mathbb B\cap\mathbb R$, thus
\eqref{malpensa190715!!!!q} can be set at another real point $\alpha\in\mathbb B$:
\begin{equation}
[R_{\alpha} f,R_{\alpha} f]_{\mathcal P}\le[(I_{\mathcal P}+\alpha R_{\alpha})f,(I_{\mathcal P}+\alpha R_{\alpha})f]_{\mathcal P}-[f(\alpha),f(\alpha)]_{\mathcal C},\quad \forall f\in\mathcal P.
\label{malpensa!!!!q}
\end{equation}

\subsection{The structure theorem}
We begin by proving that Lemma \ref{Lemma27} can be generalized to this setting in fact we have:
\begin{lemma}\label{Lemma27q}
Let $\alpha, \beta\in\mathbb R^+$. Then
equality \eqref{equadb2} holds in $\mathbf H_2(\mathbb H_+)$.
\end{lemma}

\begin{proof}
Also in the quaternionic setting, we use the fact that $\frac{1}{2\pi}k(p,q)=\frac{1}{2\pi} (p+\overline{q})^{-\star}$, where $k(p,q)$ is as in \eqref{kernel}, is the reproducing kernel of $\mathbf H_2(\mathbb H_+)$ and we prove the equality for $k(\cdot,\mu)$, $k(\cdot,\nu)$. Let us consider $\mu\in\mathbb H_+$ and $\alpha\in\mathbb R^+$. Then
\[
k(\alpha, \mu)= (|\mu|^2 +{\rm Re}(\mu)\alpha +\alpha^2)^{-1}(\alpha +\mu)=(\alpha +\mu)^{-1},
\]
moreover
\[
\begin{split}
R_\alpha k(p,&\mu)  = (p-\alpha)^{-1}(k(p,\mu)-k(\alpha,\mu) )\\
& = (p-\alpha)^{-1}( (|\mu|^2 +{\rm Re}(\mu)p + p^2)^{-1}(p +\mu)- ( \alpha- \bar \mu)^{-1} )\\
& = (p-\alpha)^{-1}(|\mu|^2 +{\rm Re}(\mu)p + p^2)^{-1}( (p +\mu)( \alpha- \bar \mu)-(|\mu|^2 +{\rm Re}(\mu)p + p^2)  ) ( \alpha- \bar \mu)^{-1}\\
& = (|\mu|^2 +{\rm Re}(\mu)p + p^2)^{-1} (p-\alpha)^{-1}(-p (p - \alpha) - (p - \alpha) \bar \mu) ( \alpha- \bar \mu)^{-1}\\
& = (|\mu|^2 +{\rm Re}(\mu)p + p^2)^{-1} (-p  - \bar \mu) ( \alpha- \bar \mu)^{-1}\\
& = -k(p,\mu)( \alpha- \bar \mu)^{-1}.\\
\end{split}
\]
It follows that
\[
\begin{split}
\langle R_\alpha \frac{1}{2\pi} k(\cdot,\mu),  \frac{1}{2\pi} k(\cdot,\nu)\rangle&=
\langle - \frac{1}{2\pi} k(\cdot,\mu)(\alpha+\overline{\mu})^{-1}, \frac{1}{2\pi} k(\cdot,\nu)\rangle\\
&= -\frac{1}{2\pi} k(\nu,\mu)(\alpha+\overline{\mu})^{-1},
\end{split}
\]
\[
\begin{split}
 \frac{1}{2\pi} k(\cdot,\mu), \langle R_\beta \frac{1}{2\pi} k(\cdot,\nu)\rangle&=
\langle  \frac{1}{2\pi} k(\cdot,\mu), -\frac{1}{2\pi} k(\cdot,\nu)(\beta+\overline{\nu})^{-1}\rangle\\
&=-({\beta}+\nu)^{-1} \frac{1}{2\pi} k(\nu,\mu),
\end{split}
\]
\[
\begin{split}
\langle R_\alpha \frac{1}{2\pi} k(\cdot,\mu), \langle R_\beta \frac{1}{2\pi} k(\cdot,\nu)\rangle&=
\langle - \frac{1}{2\pi} k(\cdot,\mu)(\alpha+\overline{\mu})^{-1}, -\frac{1}{2\pi} k(\cdot,\nu)(\beta+\overline{\nu})^{-1}\rangle\\
&=({\beta}+\nu)^{-1} \frac{1}{2\pi} k(\nu,\mu)(\alpha+\overline{\mu})^{-1},
\end{split}
\]
furthermore
\[
2\pi \overline{k(\beta,\nu)} k(\alpha, \mu)=\frac{1}{2\pi}\overline{(\beta+\overline{\nu})^{-1}}(\alpha+\overline{\mu})^{-1}= ({\beta}+\nu)^{-1}(\alpha+\overline{\mu})^{-1}.
\]
Let us now compute the left-hand side of \eqref{equadb2}, neglecting everywhere the factor $1/2\pi$ (i.e. multiplying \eqref{equadb2} by $2\pi$):
\[
\begin{split}
&2\pi(\langle R_\alpha f,g\rangle+\langle f,R_\beta g\rangle+(\alpha+\overline{\beta})
\langle R_\alpha f,R_\beta g\rangle+ 2\pi \overline{g(\beta)}f(\alpha))=\\
&=-k(\nu,\mu)(\alpha+\overline{\mu})^{-1} -({\beta}+\nu)^{-1} k(\nu,\mu)+(\alpha+\beta)({\beta}+\nu)^{-1}  k(\nu,\mu)(\alpha+\overline{\mu})^{-1}\\
&\hspace{7.5cm}+({\beta}+\nu)^{-1}(\alpha+\overline{\mu})^{-1}=\\
& =({\beta}+\nu)^{-1}[-({\beta}+\nu) k(\nu,\mu)-k(\nu,\mu)(\alpha+\overline{\mu})+ (\alpha+\beta) k(\nu,\mu)+1   ] (\alpha+\overline{\mu})^{-1}=\\
& =({\beta}+\nu)^{-1}[-\nu k(\nu,\mu)-k(\nu,\mu)\overline{\mu}+1   ] (\alpha+\overline{\mu})^{-1}.
\end{split}
\]
Proposition 4.7 in \cite{acls_milan} yields $-\nu k(\nu,\mu)-k(\nu,\mu)\overline{\mu}+1 =0$ and thus the equality
\eqref{equadb2} holds. The fact that \eqref{equadb2} holds in $\mathbf H_2(\mathbb H_+)$ follows from the fact that the linear span of the reproducing kernels form a dense subset of $\mathbf H_2(\mathbb H_+)$ and from the fact that $R_\alpha$, $R_\beta$ are bounded operators (see the proof of Lemma \ref{Lemma27}).
\end{proof}
\begin{theorem}
\label{michelle112345}
Let $\Omega\subset \mathbb H_+$ be a s-domain, and let $\alpha\in\mathbb R\cap\Omega$ and let $\mathcal C$ be a
two-sided quaternionic Pontryagin space.
Let $\mathcal P$ be a right-sided reproducing kernel Pontryagin space of $\mathcal C$-valued functions slice hyperholomorphic in $\Omega$, which is $R_\alpha$-invariant and such that inequality \eqref{malpensa190715} holds in $\mathcal P$.
Then the functions of $\mathcal P$ have a slice hypermeromorphic extension to $\mathbb H_+$ and there exists a
quaternionic two-sided Pontryagin space $\mathcal C_1$ with ${\rm ind}_-(\mathcal C_1)={\rm ind}_-(\mathcal C)$ and a function $S\in\mathcal S_\kappa(\mathcal C_1,\mathcal C)$, with $\kappa={\rm ind}_-(\mathcal P)$, such that the reproducing kernel of $\mathcal P$ is given by
\begin{equation}
\label{sarah123456789012}
2\pi \left(pK_S(p,q)+K_S(p,q)\overline{q}\right)=I_{\mathcal C}-S(p)S(q)^*
\end{equation}
\end{theorem}

\begin{proof} The proof follows that of Theorem \ref{michelle00000}, where now $\alpha$ is real and $z=t$ and $w=s$ are assumed real.
Equality \eqref{ineqfund} is still valid here, and so is the factorization \eqref{FH}, see
\cite[Proof of Theorem 7.1, STEP 3 and (7.4), p. 862]{MR3192300}.
We conclude that for $t,s\in\Omega\cap\mathbb R$,
\[
\frac{I_{\mathcal C}-S(t)S(s)^*}{2\pi(t+s)}=C_tC_s^*.
\]
The result follows by slice hyperholomorphic extension of these operator-valued functions; see Remark \ref{extension!!!!}.
\end{proof}
\begin{remark}{\rm
As in the complex case, but now for real and positive $\alpha$ and $\beta$, we have \eqref{equadb2} in the Hardy space of the half-space, that is \eqref{malpensa190715} holds as an equality there. As in Section \ref{subsec2.3} we obtain that the orthogonal of the space
$\mathcal P$ is equal to $M_S\mathbf H_2(\mathbb H_+)$, where now $M_S$ is the operator of $\star$-multiplication by $S$ on the left.
A special case was considered in \cite{acs_trend_1}. Here too, one can consider the spaces $\mathbf H_2(\mathbb H_+, J)$, where $J$ is a signature
matrix. Now $J$ is assumed to have real, rather than complex or quaternionic, coefficients.}
\end{remark}
\subsection{A unified setting}
One can define a unified setting as in the complex plane, but because of the problems arising with composition operators, it is necessary to restrict oneself with functions $a$ and $b$ in the class of intrinsic functions (see Remark \ref{rk:intrinsic}); for functions slice hyperholomorphic in a neighborhood of the origin, this means that their developments in powers of $p$ have only real coefficients.
Specifically, we will consider an open axially symmetric s-domain $\Omega\subseteq\mathbb H$, and a pair of functions $a(p)$ and $b(p)$ quaternionic intrinsic in $\Omega$ such that
\[
\Omega_+=\left\{p\in\Omega\,\,;\,\,|b(p)|<|a(p)|\right\}\quad \text{and}\quad \Omega_-=\left\{p\in\Omega\,\,;\,\,|b(p)|>|a(p)|\right\}
\]
are both nonempty. We also assume that $\Omega_+$ is an s-domain.
\begin{remark}
We note that $\Omega_+$ and $\Omega_-$ are axially symmetric, in fact for any slice hyperholomorphic function we have $f(p)=f(x+Jy)=\alpha(x,y)+J\beta(x,y)$ (see \cite{MR2752913,MR3013643})
and so $|f(p)|$ does not depend on the choice of $J$.
\end{remark}
We then define, for any $f$ slice hyperholomorphic in some open set $\Omega$
\[
\begin{split}
(R(a,b,\alpha)f)(p)&=(a(\alpha)b(p)-b(\alpha)a(p))^{-\star}\star(a(p)f(p)-a(\alpha)f(\alpha))\\
&=(a(\alpha)b(p)-b(\alpha)a(p))^{-1}(a(p)f(p)-a(\alpha)f(\alpha)), \quad \alpha\in\mathbb R.
\end{split}
\]
\begin{remark}{\rm
Easy computations show that \eqref{rarb1} holds also in the quaternionic setting, since $\alpha\in\mathbb R$:
\begin{equation}\label{rarb1q}
\begin{split}
a(\alpha)&R(b,a,\alpha)+b(\alpha)R(a,b,\alpha)
\\
&= a(\alpha)(b(\alpha)a(p)-a(\alpha)b(p))^{-1}(b(p)f(p)-b(\alpha)f(\alpha))\\
&
+b(\alpha)(a(\alpha)b(p)-b(\alpha)a(p))^{-1}(a(p)f(p)-a(\alpha)f(\alpha))\\
&=(b(\alpha)a(p)-a(\alpha)b(p))^{-1}(a(\alpha)b(p)f(p)-a(\alpha)b(\alpha)f(\alpha)-b(\alpha) a(p)f(p)+b(\alpha)a(\alpha)f(\alpha) )\\
&=-f(p).
\end{split}
\end{equation}}
\end{remark}
The following results is the quaternionic counterpart of Proposition \ref{P32}:
\begin{proposition}\label{P32q}
Let $\Omega_+$ be as above and let $\alpha\in\Omega_+$.
A function $S$ is a generalized Schur function if and only if it can be written in the form
\begin{equation}
\label{monique1q}
S(p)=H+\frac{\sigma(p)-\sigma(\alpha)}{1-\sigma(p)\overline{\sigma(\alpha)}}G\star\left(I_{\mathcal P}-\frac{\sigma(p)-\sigma(\alpha)}{1-\sigma(p)\overline{\sigma(\alpha)}}T\right)^{-\star}F,
\end{equation}
where $\mathcal P$ is a Pontryagin space with index of negativity $\kappa$ and where the operator-matrix
\eqref{s-adrs111} is coisometric.
\end{proposition}
\begin{proof} The proof closely follows the proof of Proposition \ref{P32}. Let $\sigma(p)=a(p)^{-1}b(p)$ and consider a real point $p_0$ in which $\sigma'(p)$ is nonzero. Then $\sigma$ is quaternionic intrinsic and  one-to-one from a neighborhood $U_{p_0}$ of $p_0$ to $\sigma(U_{p_0})\subseteq\mathbb B$. Without loss of generality, we can assume that $U_{p_0}$ is axially symmetric s-domain (so that also $\sigma(U_{p_0})$ is an axially symmetric s-domain). Let $\phi:\, \sigma(U_{p_0})\to U_{p_0}$ be the inverse of $\sigma$. The function $\phi$ is quaternionic intrinsic, so we can define the function $M(p)=S(\phi(p))$ where $p\in\sigma(U_{p_0})$. Thus $S(p)=M(\sigma(p))$, where $M$ is a  $\mathbf L(\mathcal D,\mathcal C)$-valued  generalized Schur function of the open unit ball which is slice hyperholomorphic in an open set $\Omega_M\subset\mathbb B$ which we can assume to be an axially symmetric s-domain. The kernel $K_M(p,q)=(I-M(p)M(q)^*)\star(1-p\overline{q})^{-\star}$ has the same number $\kappa$ of negative squares as $K_S(p,q)$. This fact can be proved as the analog result in the complex case. Let now
$u\in\mathbb B\cap\mathbb R$ be a point in $\Omega_M$, and let $b_u(p)=(1+p{u})^{-1}(p+u)$. The function $M_u(p)=M(b_u(p))$ is a generalized Schur function of the open unit ball which is slice hyperholomorphic in a neighborhood of the origin. By
\cite[Theorem 7.1, p. 862]{MR3192300}, $M$ can be written as $M(b_u(p))=H+pG\star (I_{\mathcal P}-pT)^{-\star}F$, where the space $\mathcal P$ and the operators
$T,F,G,H$ are as in \eqref{s-adrs-123}. We can take $u=\sigma(\alpha)$ since $S(p)=M(\sigma(p))$ is slice hyperholomorphic in an neighborhood of $\alpha$.
Replacing $p$ by $b_{-\sigma(\alpha)}(p)$ we have
\[
M(p)=H+\frac{p-\sigma(\alpha)}{1-p{\sigma(\alpha)}}G\star\left(I_{\mathcal P}-
\frac{p-\sigma(\alpha)}{1-p{\sigma(\alpha)}}T\right)^{-\star}F,
\]
(note that since $\sigma(\alpha)\in\mathbb R$ we are allowed to write a quotient instead of
$(1-p{\sigma(\alpha)})^{-1}(p-\sigma(\alpha)$).
The result follows by replacing $p$ by $\sigma(p)$.
\end{proof}

Using the notation introduced above, we can then prove the analog of Proposition \ref{P33}:
\begin{proposition}\label{P33q}
Let $S$ be a  $\mathbf L(\mathcal D,\mathcal C)$-valued  generalized Schur function, slice hyperholomorphic in some open subset of $\Omega_+\subset\Omega$, and with associated reproducing kernel Pontryagin space
$\mathcal P(S)$. Then, $\mathcal P(S)$ is $R(a,b,\alpha)$-invariant and the inequality
\begin{equation}
\label{ineq1234567!!!q}
[R(a,b,\alpha)f, R(a,b,\alpha)f]_{\mathcal P(S)}-[R(b,a,\alpha)f, R(b,a,\alpha)f]_{\mathcal P(S)}+[f(\alpha),f(\alpha)]_{\mathcal C}\le0
\end{equation}
holds in it.
\end{proposition}
\begin{proof}
Let $M$ be as in the proof of the preceding proposition. Then, since $a,b$ are quaternionic intrinsic, and by the validity of
\eqref{malpensa!!!!q} (see e.g. Corollary 8.3.9 in \cite{acs_book}) and of \eqref{rarb1q}, we can repeat all the computations in the proof of Proposition \ref{P33}. \end{proof}

We then have the structure theorem:

\begin{theorem}
Let $\mathcal C$ be a Pontryagin space, and let $\Omega$ be an open axially symmetric s-domain in $\mathbb H$.
Let $\mathcal P$ be a reproducing kernel Pontryagin space of $\mathcal C$-valued functions analytic in $\Omega$, which is $R_0$-invariant and such that \eqref{ineq1234567!!!} holds in $\mathcal P$. Then every elements of $\mathcal P$ has a unique slice hypermeromorphic extension to $\Omega_+$, and there exists a Pontryagin space $\mathcal C_1$ with ${\rm ind}_-(\mathcal C_1)={\rm ind}_-(\mathcal C)$ and a function $S\in\mathcal S_\kappa(\mathcal C_1,\mathcal C)$, with $\kappa={\rm ind}_-(\mathcal P)$, such that the reproducing kernel of the space $\mathcal P$ is of the form \eqref{michelle12345678!!!!!!!!!!!!}.
\label{michelle00000}
\end{theorem}
\begin{proof}
Also the proof of this result is obtained by mimicking the arguments to prove Theorem
\ref{michelle00000}. Note that Step 1 can be repeated by virtue of Proposition \ref{P33q}. Since the denominators are quaternionic intrinsic functions and since Proposition \ref{P32q} holds, the computations in Step 2, 3 and 4 can be repeated by formally replacing the complex variable $z$ by the quaternion $p$.
\end{proof}

{\bf Acknowledgments:} It is a pleasure to thank Professors Vladimir Bolotnikov, Aad Dijksma and Jim Rovnyak for comments on an earlier
version of this paper.

\end{document}